\documentclass[12pt]{article}
\usepackage[utf8]{inputenc}

\usepackage{geometry}
\usepackage{amsmath, amssymb, amsthm, mathtools}
\usepackage[all]{xy}
\usepackage{caption, subcaption}
\usepackage{xcolor}
\usepackage{comment}
\usepackage{authblk}

 \newcommand{\C}{\mathbb{C}}

\newcommand{\V}{\mathcal{V}}
\newcommand{\M}{\mathcal{M}}

\newcommand{\Eq}{\mathrm{Eq}}

\renewcommand{\lim}{\mathrm{lim}}

  \newcommand{\CPos}{\mathsf{CPos}} \newcommand{\Pos}{\mathsf{Pos}}

\newcommand{\Lat}{\mathsf{Lat}}
\newcommand{\Sub}{\mathsf{Sub}}

\newcommand{\Mon}{\mathsf{Mon}}

\newcommand{\op}{\mathrm{op}}
\renewcommand{\:}{\colon}

\newcommand{\ExtFirst}{\ensuremath{\mathsf{E1}}}
\newcommand{\ExtSecond}{\ensuremath{\mathsf{E2}}}

\newcommand{\ExtFirstop}{\ensuremath{\mathsf{C1}}}
\newcommand{\ExtSecondop}{\ensuremath{\mathsf{C2}}}

\usepackage{pigpen}
\newcommand{\pol}{\ar@{}[ur]|<<{\text{\pigpenfont G}}}
\newcommand{\pbl}{\ar@{}[dr]|<<{\text{\pigpenfont A}}}
\newcommand{\por}{\ar@{}[ul]|<<{\text{\pigpenfont I}}}
\newcommand{\pbr}{\ar@{}[dl]|<<{\text{\pigpenfont C}}}

\newtheorem{theorem}{Theorem}[section]
\newtheorem{proposition}[theorem]{Proposition}
\newtheorem{corollary}[theorem]{Corollary}
\newtheorem{lemma}[theorem]{Lemma}
  \theoremstyle{definition}
\newtheorem{definition}[theorem]{Definition}
\newtheorem{remark}[theorem]{Remark}

\newtheorem{notation}[theorem]{Notation}

\usepackage{tikz}

\usetikzlibrary{cd}

\title{On extensivity of morphisms }

 \author[1,2]{Michael Hoefnagel}
 \author[1,2]{Emma Theart}
 \affil[1]{\small{\textit{Mathematics Division, Department of Mathematical Sciences, Stellenbosch University, Private Bag X1 Matieland 7602, South Africa}}}

  \affil[2]{\small{\textit{National Institute
for Theoretical and Computational Sciences (NITheCS), South Africa}}}

\date{}

\begin{document}

\maketitle

\begin{abstract}
Extensivity of a category \cite{Carboni} may be described as a property of coproducts in the category, namely, that they are disjoint and universal. An alternative viewpoint is that it is a property of morphisms in a category. This paper explores this point of view through a natural notion of extensive and coextensive morphism. Through these notions, topics in universal algebra, such as the strict refinement and Fraser-Horn properties, take categorical form and thereby enjoy the benefits of categorical generalisation. On the other hand, the universal algebraic theory surrounding these topics inspire categorical results. One such result we establish in this paper is that a Barr-exact category $\C$ is coextensive if and only if every split monomorphism in $\C$ coextensive. 
\end{abstract}
\textbf{keywords}: Extensive category, extensive morphism, strict refinement property, Fraser-Horn property 

\section{Introduction} \label{section: introduction}
An \emph{extensive category} \cite{Carboni} is, informally speaking, one in which finite coproducts (sums) exist and behave in a manner similar to disjoint unions of sets. Formally, given a category $\C$ with finite coproducts, then $\C$ is extensive if the canonical functor
\[
(\C \downarrow X_1) \times (\C \downarrow X_2) \xrightarrow{+} (\C \downarrow (X_1 + X_2))
\]
is an equivalence for any objects $X_1$ and $X_2$. Equivalently, extensivity may be formulated internally: a category $\C$ with finite coproducts is extensive if $\C$ admits all pullbacks along coproduct injections, and for any diagram
\begin{equation} \label{diagram: intro}
\xymatrix{
      A_1 \ar[r] \ar[d] &  A \ar[d]^f & A_2 \ar[l] \ar[d] \\
      X_1 \ar[r] & X & X_2 \ar[l]
}
\end{equation}
where the bottom row is a coproduct, the top row is a coproduct if and only if the squares are pullbacks (see Proposition~2.2 in \cite{Carboni}). These two formulations suggest two distinct points of view on the nature of extensivity: the first emphasises it as a property of coproducts, while the second as a property of morphisms. This paper focuses on the second viewpoint and, to that end, defines a morphism $f: A \to X$ in a category $\C$ with finite coproducts to be \emph{extensive} if every coproduct injection into $X$ admits a pullback along $f$, and $f$ satisfies the diagrammatic property described above for any diagram such as \eqref{diagram: intro}. A category with finite coproducts is then extensive if and only if every morphism is extensive.

The prototypical example of an extensive category is the category of sets, yet the category of pointed sets is not extensive. For instance, if $f\:A \to X$ is an extensive morphism in a pointed category, the left-hand square in 
\[\xymatrix{
      0 \ar[r] \ar[d] &  A\ar[d]^f & A  \ar[l]\ar[d] \\
      0  \ar[r] & X  & X \ar[l]}\]
being a pullback forces $\ker(f) = 0$, since both the top and bottom rows are coproducts. This illustrates that not every morphism of pointed sets is extensive. In fact, in the category of pointed sets, the morphisms $f$ with $\ker(f) = 0$ are precisely the extensive morphisms (see Proposition~73 in \cite{TheartMSc}). Dually, the category of groups is not coextensive, however, every product projection $p: G \to X$ where $G$ is a \emph{centerless} (or \emph{perfect}) group is coextensive. As will be shown in Section~\ref{section: coextensivity of product projections}, this is a categorical formulation of the fact that centerless groups have the \emph{strict refinement property} in the sense of \cite{ChangJonssonTarski1964}. In the category $\Lat$ of lattices (and lattice homomorphisms), every surjective homomorphism is coextensive, which may be seen as a categorical formulation of the fact that lattices have the \emph{Fraser-Horn} property \cite{FraserHorn1970}. 

Extensivity of a morphism is the conjunction of two separate properties, namely, the properties ($\ExtFirst$) and ($\ExtSecond$) described in section~\ref{section: general results}. We note here that for a morphism $f$ to satisfy the condition ($\ExtFirst$) is equivalent to the condition that $f$ be \emph{crisp} in the sense of \cite{PredicateTransformer} (see Definition 3 therein). Moreover, a category $\C$ with finite coproducts is a \emph{Boolean category} in sense of \cite{PredicateTransformer} if and only if every coproduct inclusion is extensive, and coproduct inclusions in $\C$ are pullback stable. 

 There is an interesting interplay between the universal algebraic and the categorical. On the one hand, the theory of coextensive morphisms allows for generalisations of universal algebraic topics to the categorical level. For instance, a variety $\V$ has the strict refinement property if and only if every product projection in $\V$ is coextensive. Expressing these algebraic concepts categorically yields natural consequences motivated  from the theory of (co)extensive categories. On the other hand, the framework of universal algebra surrounding these algebraic properties suggest corresponding categorical results. For instance,  Theorem~\ref{proposition: Barr-exact coextensive} shows that every Barr-exact \cite{BarrGrilletOsdol1971} category $\C$ with global support is coextensive if and only if every split monomorphism is coextensive.

\section{General results} \label{section: general results}
Let us restate the definition of extensive morphism from the introduction. For any category $\C$, we will be concerned with two conditions on a morphism $f\:A \to X$ in $\C$:
\begin{enumerate}
    \item[(\ExtFirst)] $f$ admits pullbacks along the injections of any coproduct diagram $$X_1 \rightarrow X \leftarrow X_2$$
   and the resulting two pullbacks 
   \begin{equation}
    \label{diagram: extensive morphism definition}
      \xymatrix{A_1 \ar[r] \ar[d] &  A\ar[d]^{f} & A_2  \ar[l]\ar[d] \\
      X_1  \ar[r] & X  & X_2 \ar[l]}
    \end{equation}
    form a coproduct diagram  $A_1 \rightarrow A \leftarrow A_2$.
    
    \item[(\ExtSecond)] for any commutative diagram
    \[\xymatrix{
      A_1 \ar[r] \ar[d] &  A\ar[d]^f & A_2  \ar[l]\ar[d] \\
      X_1  \ar[r] & X  & X_2 \ar[l]}\]
    in $\C$ where the bottom row is a coproduct diagram, if the top row is a coproduct diagram, then both squares are pullbacks.
\end{enumerate}
Then we may restate the main definition of this paper. 
\begin{definition} \label{definition: extensive morphism}
A morphism $f\:A\rightarrow X$ in a category $\C$ said to be \emph{extensive} in $\C$ if it satisfies $(\ExtFirst)$ and $(\ExtSecond)$. Dually, $f$ is called \emph{coextensive} in $\C$ if it is extensive in $\C^{\op}$. 
\end{definition}
In what follows we will refer to the category-theoretic duals $(\ExtFirstop)$ and $(\ExtSecondop)$ of the conditions $(\ExtFirst)$ and $(\ExtSecond)$, respectively. Thus, the morphism $f$ is coextensive if and only if $f$ satisfies both of the following:
\begin{enumerate}
    \item[(\ExtFirstop)] $f$ admits pushouts along the productions of any product diagram $$A_1 \leftarrow A \rightarrow A_2$$
   and the resulting two pushouts 
   \begin{equation}
    \label{diagram: extensive morphism definition}
      \xymatrix{A_1  \ar[d] &  A \ar[d]^{f} \ar[r] \ar[l] & A_2 \ar[d] \\
      X_1  & X \ar[r] \ar[l] & X_2 }
    \end{equation}
    form a product diagram $X_1 \leftarrow X \rightarrow X_2$.
    
    \item[(\ExtSecondop)] for any commutative diagram
    \[\xymatrix{A_1  \ar[d] &  A \ar[d]^{f} \ar[r] \ar[l] & A_2 \ar[d] \\
      X_1  & X \ar[r] \ar[l] & X_2 }\]
    in $\C$ where the top row is a product diagram, if the bottom row is a product diagram, then both squares are pushouts.
\end{enumerate}

\begin{proposition}
	\label{proposition: composite of extensive}
	In any category $\C$, the composite of two extensive morphisms is extensive.
\end{proposition}
\begin{proof}
	Let $f\colon X \rightarrow Y$ and $g\colon Y \rightarrow Z$ be extensive morphisms in $\C$. Suppose we have a coproduct diagram $Z_1 \xrightarrow{z_1} Z \xleftarrow{z_2} Z_2$.
    By extensivity of $g$, $g$ admits pullbacks along $z_1$ and $z_2$, and these pullback squares together form a coproduct diagram $Y_1 \xrightarrow{y_1} Y \xleftarrow{y_2} Y_2$ in the top row. Since $f$ is extensive, it in turn admits pullbacks along $y_1$ and $y_2$. These pullback diagrams are shown in (\ref{diag-composite1}). 
	\begin{equation} \label{diag-composite1}
		\xymatrix{
			X_1 \pbl \ar[r]^{x_1} \ar[d] &  X \ar[d]^{f} & X_2 \pbr  \ar[l]_{x_2} \ar[d] \\
			Y_1 \pbl \ar[r]^{y_1} \ar[d] &  Y \ar[d]^{g} & Y_2 \pbr \ar[l]_{y_2}\ar[d] \\
			Z_1 \ar[r]_{z_1} & Z & \ar[l]^{z_2} Z_2}
	\end{equation}
	By pullback pasting, it follows that these composite squares in (\ref{diag-composite1}) form pullback diagrams of $gf$ along $z_1$ and $z_2$, where the top row is a coproduct. Thus $gf$ satisfies $(\ExtFirst)$. Next, suppose we have a commutative diagram
	\begin{equation} \label{diag-composite2}
		\xymatrix{ X_1 \ar[r]^{x_1} \ar[d]_{j_1} & X \ar[d]_{g  f} & \ar[l]_{x_2} X_2 \ar[d]^{j_2} \\
			Z_1 \ar[r]_{z_1} & Z & \ar[l]^{z_2} Z_2}
	\end{equation}
	in $\C$ where the top and bottom rows are coproduct diagrams. We form the pullbacks of $g$ along $z_1$ and $z_2$, and these pullbacks induce the dotted morphisms in (\ref{diag-composite3}). Since the top and middle rows of (\ref{diag-composite3}) are coproduct diagrams, it follows by the extensivity of $f$ that the top two squares are pullbacks. By pullback pasting we conclude that (\ref{diag-composite2}) consists of a pair of pullback squares. Hence, $gf$ satisfies $(\ExtSecond)$.

	\begin{equation} \label{diag-composite3}
		\xymatrix{
			X_1  \ar@{}[dr]  \ar@/_1pc/[dd]_{j_1}  \ar[r]^{x_1} \ar@{.>}[d] &  X \ar[d]^{f} & X_2  \ar@{}[dl]\ar@/^1pc/[dd]^{j_2}   \ar[l]_{x_2} \ar@{.>}[d] \\
			Y_1 \pbl \ar[r]^{p_1} \ar[d] &  Y \ar[d]^{g} & Y_2 \pbr \ar[l]_{p_2}\ar[d] \\
			Z_1 \ar[r]_{z_1} & Z & \ar[l]^{z_2} Z_2}
	\end{equation}
\end{proof}
\begin{corollary} \label{corollary: C extensive iff split epis and inclusions are}
	If $\C$ is a category with binary coproducts, then $\C$ is extensive if and only if every split epimorphism and every coproduct inclusion is extensive. 
\end{corollary}

\begin{proof}
	This follows from the fact that every morphism $f\:X \to Y$ in such a category $\C$ factors as 
	$\xymatrix{
		X \ar[r]_-{\iota_1} & X + Y \ar[r]_-{\langle f, 1_Y \rangle} & Y
	}$
 where ${\langle f, 1_Y \rangle}$ is morphism induced by $1_Y$ and $f$. 
	
\end{proof}

\begin{lemma}
  \label{lemma: if the squares exist}
    For any  morphisms $f\colon X \rightarrow Y$ and $g\colon Y \to Z$ in any category $\C$, if $gf$ is extensive and for each coproduct diagram $Y_1 \xrightarrow{y_1} Y \xleftarrow{y_2} Y_2$
    there exists a pair of pullback squares
    \begin{equation} 
      \label{diagram: the square that exists}
      \xymatrix{Y_1 \ar[r]^{y_1} \ar[d]_{g_1}& Y \ar[d]^{g}& Y_2 \ar[l]_{y_2} \ar[d]^{g_2} \\
      Z_1 \ar[r]_{z_1}& Z & \ar[l]^{z_2} Z_2}
    \end{equation}
    where the top and bottom rows are coproduct diagrams, then $f$ is extensive.
\end{lemma}

\begin{proof}
  We show that $f$ satisfies $(\ExtFirst)$. Let $Y_1 \xrightarrow{y_1} Y \xleftarrow{y_2} Y_2$ be any coproduct diagram. Then, the squares in (\ref{diagram: the square that exists}) exist by assumption, and therefore, because $gf$ is extensive, we may pull $gf$ back along $z_1$ and $z_2$ to form the diagram of solid arrows below
  \[\xymatrix{X_1 \ar@{.>}[d] \ar[r] \ar@/_2em/[dd] & \ar[d]^{f} X & X_2 \ar@/^2em/[dd] \ar[l] \ar@{.>}[d]\\ 
    Y_1 \ar[r]^{y_1} \ar[d]_{g_1}& Y \ar[d]^{g}& Y_2 \ar[l]_{y_2} \ar[d]^{g_2} \\
    Z \ar[r]_{z_1}& Z & \ar[l]^{z_2} Z_2}\]
  where the dotted arrows are the morphisms induced by these pullbacks, and the top row is a coproduct diagram by extensivity of $gf$. It then follows that, since the outer and bottom squares are pullbacks, the top two squares are also pullbacks. Hence, $f$ satisfies $(\ExtFirst)$.

  To show that $f$ satisfies $(\ExtSecond)$, suppose we have a commutative diagram
  \[\xymatrix{
    X_1 \ar[d]_{f_1} \ar[r]^{x_1} & X \ar[d]^f & \ar[l]_{x_2} X_2 \ar[d]^{f_2} \\
    Y_1\ar[r]_{y_1} & Y & Y_2 \ar[l]^{y_2}}\]
  where the top and bottom rows are coproducts. Then, since $gf$ is extensive, it follows that the outer squares in
  \[\xymatrix{
    X_1 \ar[d]_{f_1} \ar[r]^{x_1} & X \ar[d]^f & \ar[l]_{x_2} X_2 \ar[d]^{f_2}\\
    Y_1 \ar[r]^{y_1} \ar[d]_{g_1}& Y \ar[d]^{g}& Y_2 \ar[l]_{y_2} \ar[d]^{g_2} \\
    Z \ar[r]_{z_1}& Z & \ar[l]^{z_2} Z_2} \]
  are pullbacks. Hence, so are the top two squares. Therefore, $f$ is extensive.
\end{proof}
\begin{lemma}
  In any category $\C$, if $\iota\colon Y \to Z$ is a coproduct inclusion which satisfies $(\ExtSecond)$, then for any coproduct diagram $Y_1 \xrightarrow{y_1} Y \xleftarrow{y_2} Y_2$,
  there exists a pair of pullback squares
    \begin{equation} 
      \xymatrix{Y_1 \ar[r]^{y_1} \ar[d]_{\iota_1}& Y \ar[d]^{\iota}& Y_2 \ar[l]_{y_2} \ar[d]^{\iota_2} \\
      Z_1 \ar[r]_{z_1}& Z & \ar[l]^{z_2} Z_2}
    \end{equation}
    where the top and bottom rows are coproduct diagrams.
\end{lemma}

\begin{proof}
   Since $\iota$ is a coproduct inclusion, it has a complementary inclusion $\iota':Y' \to Z$  making $Y \xrightarrow{\iota} Z \xleftarrow{\iota'} Y'$
  a coproduct diagram. We may then form the coproduct diagram 
  \begin{equation*} 
    \xymatrix{Y_2 \ar[r]^-{j_1} & Y_2 + Y' & \ar[l]_-{j_2}  Y'}.
  \end{equation*}
  Consider the commutative diagram:
  \begin{equation}
    \label{diagram: inclusions}
    \xymatrix{ 
    Y_1 \ar[r]^{y_1} \ar[d]_{1_{Y_1}} &  Y \ar[d]_{\iota} & Y_2  \ar[l]_{y_2}\ar[d]^{j_1} \\
    Y_1  \ar[r]_-{\iota y_1} & Z  & Y_2 + Y' \ar[l]^-{\langle \iota y_2, \iota'\rangle }
    }
  \end{equation}
  It is routine to verify that the bottom row is a coproduct, and hence since $\iota$ satisfies $(\ExtSecond)$, it now follows that both squares in \eqref{diagram: inclusions} are pullbacks.
\end{proof}

From the above two lemmas, we get the following proposition.

\begin{proposition} \label{proposition: if extensive implies f extensive}
  Let $\C$ be a category and $f\colon A \rightarrow B$ any morphism in $\C$. Let $\iota\colon B \rightarrow X$ be any coproduct inclusion in $\C$. Then, if $\iota f$ is extensive and $\iota$ satisfies $(\ExtSecond)$ it follows that $f$ is extensive.
\end{proposition}

\begin{corollary}
Let $\C$ be any category and $\iota:A \to X$ any coproduct inclusion in $\C$. If every coproduct inclusion of $X$ is extensive, then every coproduct inclusion of $A$ is extensive. 
\end{corollary}

\subsection{Coextensivity of identity morphisms}
Throughout this section we will fix a category $\C$. If $f\: A \rightarrow X$ is any isomorphism in $\C$ then for any morphism $g\:A \to B$, the square 
\[
\xymatrix{
	A \ar[r]^g \ar[d]_f& B \ar[d]^{1_B} \\ 
	X \ar[r]_{gf^{-1}}& B
}
\]
is a pushout in $\C$, and hence the proposition below is immediate.  
\begin{proposition}
	\label{proposition: isomorphisms satisfy E1 and C1}
 Every isomorphism in $\C$ satisfies $(\ExtFirstop)$ and hence also $(\ExtFirst)$ by duality.
\end{proposition}  

If $f: A \rightarrow X$ and $g: B\rightarrow Y$ are isomorphisms in $\C$, then any commutative diagram
\[\xymatrix{A \ar[r] \ar[d]_{f}& B \ar[d]^{g} \\
X \ar[r]& Y}\]
is a pushout, leading to the following proposition.
\begin{proposition} \label{proposition: isomorphisms satisfy C2 iff}
    Each isomorphism in $\C$ satisfies $(\ExtSecondop)$ if and only if the following condition holds: for any pair of morphisms $f_1$ and $f_2$ in $\C$, if their product $f_1\times f_2$ is an isomorphism, then $f_1$ and $f_2$ are isomorphisms.
\end{proposition}

\begin{proposition}
  \label{proposition: identity coextensive implies (co-crisp implies coextensive)}
  For any morphism $f \:A \to B$ in $\C$, if $f$ satisfies $(\ExtFirstop)$ and $1_B$ satisfies $(\ExtSecondop)$ then $f$ is extensive. 
\end{proposition}
\begin{proof}
  Suppose $f\colon A\to B$ satisfies ($\ExtFirstop$).  Consider the following commutative diagram
  \[
\xymatrix{
  A_1 \ar[d]_{f_1} & A \ar[l]_{a_1} \ar[r]^{a_2} \ar[d]_{f} & A_2 \ar[d]^{f_2} \\
  B_1 & B \ar[l]^{b_1} \ar[r]_{b_2} & B_2
}
\]
  where the rows are products.  Since $f$ satisfies $(\ExtFirstop)$, the pushouts of $f$ along $a_1$ and $a_2$ exist.  Suppose that the top two squares in the following diagram are these pushouts of $f$

  \[
  \xymatrix{A_1 \ar@/_1.2em/[dd]_{f_1} \ar[d] & \ar[l]_{a_1} \ar[d]^f A \ar[r]^{a_2} & A_2 \ar@/^1.2em/[dd]^{f_2}\ar[d]   \\
  P_1 \pol \ar@{.>}[d]^{p_1} & \ar[l] \ar[d]^{1_B} B \ar[r] & P_2 \por\ar@{.>}[d]_{p_2}\\
  B_1 & \ar[l]^{b_1} B \ar[r]_{b_2}& B_2}
  \]
 so that the middle row is a product diagram.  Since the top two squares are pushouts, there exist morphisms $p_1$ and $p_2$ such that the diagram is commutative.   Since the middle and bottom rows are products and $1_B$ satisfies $(\ExtSecondop)$, the bottom two squares are pushouts.  Pasting pushouts together, we have that the squares in the initial diagram are pushouts.  So $f$ satisfies $(\ExtSecondop)$, and is therefore coextensive.
\end{proof}

Applying the above results and the corresponding dual results, we obtain the following.

\begin{corollary} \label{corollary: Ext1 are extensive}
If every identity morphism in $\C$ is extensive/coextensive, a morphism $f\:X\to Y$ is extensive/coextensive if and only if it satisfies $(\ExtFirst)/(\ExtFirstop)$.     
\end{corollary}

\begin{corollary} \label{corollary: identity ext iff isos ext}
In any category $\C$, every isomorphism is extensive/coextensive if and only if every identity morphism in $\C$ extensive/coextensive. 
\end{corollary}

Not every category has (co)extensive identity morphisms. To illustrate this, consider the partially ordered set $\{0,1\}$ viewed as category. Then the left-hand square in the diagram 
\[
\xymatrix{
0 \ar[d] & 0 \ar[d] \ar[r] \ar[l] & 0 \ar[d] \\ 
1 & 0 \ar[r] \ar[l]& 0
}
\]
is not a pushout. Consequently, the identity morphism on $0$ does not satisfy $(\ExtSecondop)$.

The proof of the lemma below is routine and left to the reader. 
\begin{lemma} \label{lemma: product diagrams lift via monomorphisms}
Given any commutative diagram
\[
\xymatrix{
  B_1 \ar@{>->}[]+<0ex,-2.2ex>;[d]_-{m_1} & A \ar[l]_{q_1} \ar[r]^{q_2} \ar[d]_{1_A} & B_2 \ar@{>->}[]+<0ex,-2.2ex>;[d]^-{m_2} \\
  A_1 & A \ar[l]^{p_1} \ar[r]_{p_2} & A_2
}
\]
  in $\C$ where $m_1$ and $m_2$ are monomorphisms, if the bottom row is a product diagram, then so is the top row. 
\end{lemma}

Recall that a morphism $e\colon A \to B$ in $\C$ is called an \emph{extremal epimorphism} if for any composable pair of morphisms $i$ and $m$ in $\C$, if $e = mi$ and $m$ is a monomorphism, then $m$ is an isomorphism. Note that if $\C$ has equalisers, then every extremal epimorphism is an epimorphism.  

\begin{lemma} \label{lemma: f_1 f_2 mono}
  Let $\C$ be a finitely complete category where every product projection in $\C$ is an epimorphism. Then, for any two morphisms $f_1$ and $f_2$ in $\C$, if their product $f_1 \times f_2$ is a monomorphism then $f_1$ and $f_2$ are monomorphisms. 
\end{lemma} 
\begin{proof}
Let $f_1\:A_1 \to B_1$ and $f_2\:A_2 \to B_2$ be any two morphisms in $\C$ such that $f_1 \times f_2$ is a monomorphism. 
  Form the kernel pairs of $f_1 \times f_2$,  $f_1$ and $f_2$ as below, with $\alpha_1$ and $\alpha_2$ induced by $K_1$ and $K_2$ respectively.
  \[\xymatrix{
    K_1 \ar@<0.5ex>[d]^-{k_2} \ar@<-0.5ex>[d]_-{k_1} & \ar@{.>}[l]_{\alpha_1} \ar@{.>}[r]^{\alpha_2} K \ar@<0.5ex>[d]^-{\ell_2} \ar@<-0.5ex>[d]_-{\ell_1} & K_2 \ar@<0.5ex>[d]^-{m_2} \ar@<-0.5ex>[d]_-{m_1}\\
  A_1 \ar[d]_{f_1} &\ar[l]_{\pi_1} A_1 \times A_2  \ar@{>->}[]+<0ex,-2.2ex>;[d]^-{f_1 \times f_2} \ar[r]^{\pi_2} & A_2 \ar[d]^{f_2}\\
  B_1  & \ar[l]^-{\mu_1} B_1 \times B_2  \ar[r]_-{\mu_2} & B_2}\]
  Since $f$ is monic, $\ell_1 = \ell_2$. Furthermore, the top row is a product since limits commute with limits. Then, from $\alpha_1$ and $\alpha_2$ being epimorphisms, we have $k_1 = k_2$ and $m_1 = m_2$ so that $f_1$ and $f_2$ are monomorphisms.
\end{proof}

\begin{proposition} \label{proposition: coextensivity of identity = product projections are extremal}
For any object $A$ in $\C$, if the identity morphism $1_A$ is coextensive, then every product projection of $A$ is an extremal epimorphism. If $\C$ has kernel pairs then the converse holds.
\end{proposition}
\begin{proof}
Suppose that $1_A$ is coextensive, and that $A_1 \xleftarrow{p_1} A \xrightarrow{p_2} A_2$
is a product diagram. Let $m_1q_1 = p_1$ be a factorisation of $p_1$ where $m_1$ is a monomorphism as in the diagram
 \[
\xymatrix{
  I_1 \ar@{>->}[]+<0ex,-2.2ex>;[d]_-{m_1} & A \ar[l]_{q_1} \ar[r]^{p_2} \ar[d]_{1_A} & A_2 \ar@{>->}[]+<0ex,-2.2ex>;[d]^-{1_{A_2}} \\
  A_1 & A \ar[l]^{p_1} \ar[r]_{p_2} & A_2
}
\]
By Lemma~\ref{lemma: product diagrams lift via monomorphisms}, it follows that the top row is a product diagram and hence that the squares are pushouts, so that $m_1$ is an isomorphism. 
 Conversely, suppose that $\C$ has kernel pairs and that product projections of $A$ are extremal epimorphims. Since all identity morphisms satisfy $(\ExtFirstop)$ by Proposition~\ref{proposition: isomorphisms satisfy E1 and C1}, we prove that identity morphisms satisfy $(\ExtSecondop)$. Consider any commutative diagram
\[
\xymatrix{
  A_1 \ar[d]_{f_1} & A \ar[l]_{p_1} \ar[r]^{p_2} \ar[d]^{1_A} & A_2 \ar[d]^{f_2} \\
  B_1 & A \ar[l]^{q_1} \ar[r]_{q_2} & B_2
}
\]
  where the rows are products.  Since $q_1$ is an extremal epimorphism and $f_1p_1=q_1$, $f_1$ is an extremal epimorphism.  Similarly, $f_2$ is an extremal epimorphism.  Consider the kernel pairs of $f_1$, $1_A$ and $f_2$:
 \[
\xymatrix{
  K_1 \ar@<0.5ex>[d]^-{k_{11}} \ar@<-0.5ex>[d]_-{k_{12}} & A \ar@{.>}[l]_-{\alpha_1} \ar@{.>}[r]^-{\alpha_2} \ar@<0.5ex>[d]^-{1_A} \ar@<-0.5ex>[d]_-{1_A} & K_2 \ar@<0.5ex>[d]^-{k_{21}} \ar@<-0.5ex>[d]_-{k_{22}} \\
  A_1 & A \ar[l]^{p_1} \ar[r]_{p_2} & A_2
}
\]
  The morphisms $\alpha_1$ and $\alpha_2$ are induced by the commutativity of the squares in the previous diagram and satisfy $k_{ij}\alpha_i=p_i$ for all $i,j\in\{1,2\}$. Since limits commute with limits, the top row of this diagram is a product. Therefore, $\alpha_1$ and $\alpha_2$ are both extremal epimorphisms and hence epimorphisms.  Consequently, $k_{11}=k_{12}$ and $k_{21}=k_{22}$.  Thus, $f_1$ and $f_2$ are monomorphims.  Since they are also extremal epimorphisms, they are isomorphisms.  Consequently, both squares in the original diagram are pushouts.  
\end{proof}

\begin{corollary} \label{corollary: identitys are extensive in finitely complete pointed} If $\C$ is finitely complete, the product projections of $\C$ are extremal epimorphisms if and only if every identity morphism in $\C$ is coextensive. \end{corollary}

\begin{corollary} If $\C$ is a pointed category, every morphism satisfying $(\ExtFirstop)$ is coextensive.
\end{corollary}

\begin{proof} Every product projection in a pointed category $\C$ is a split epimorphism and therefore an extremal epimorphism. Thus, by Proposition~\ref{proposition: coextensivity of identity = product projections are extremal}, all identity morphisms in $\C$ are coextensive. Therefore, the result follows from Corollary~\ref{corollary: Ext1 are extensive}. \end{proof}

Extensivity of a category with finite coproducts is expressed in terms of the functor $+$. The following proposition provides a similar formulation for the extensivity of identity morphisms.

\begin{proposition}
  Let $\C$ be a category with binary coproducts. Identity morphisms are extensive if and only if the functor $$+: (\C \downarrow X_1) \times (\C\downarrow X_2) \longrightarrow (\C \downarrow (X_1+X_2))$$ is conservative for all objects $X_1$ and $X_2$ in $\C$.
\end{proposition} 

\begin{proof}
This result can be derived from the following series of equivalent statements:
\begin{enumerate} \renewcommand*\labelenumi{(\theenumi)}
    \item The functor $+$ is conservative for all objects $X_1$ and $X_2$;
    \item Each pair of morphisms $f_1$ and $f_2$ are isomorphisms whenever $f_1+f_2$ is an isomorphism;
    \item All isomorphisms satisfy $(\ExtSecond)$;
    \item All isomorphisms are extensive;
    \item All identity morphisms are extensive.
\end{enumerate}
The equivalence $(1)\Leftrightarrow (2)$  follows readily, $(2) \Leftrightarrow (3)$ is given (in dual form) in Proposition~\ref{proposition: isomorphisms satisfy C2 iff},
 $(3) \Leftrightarrow (4)$ follows from Proposition~\ref{proposition: isomorphisms satisfy E1 and C1}, while $(4) \Leftrightarrow (5)$ is established in Corollary~\ref{corollary: identity ext iff isos ext}.
\end{proof}

The results of this section culminate in the following corollary.
\begin{corollary}
The following are equivalent for a finitely complete category $\C$:
\begin{itemize}
    \item Every product projection in $\C$ is an extremal epimorphism;
    \item Every identity morphism in $\C$ is coextensive;
    \item The functor $$\times: (\C \downarrow X_1) \times (\C\downarrow X_2) \longrightarrow (\C \downarrow (X_1\times X_2))$$ is conservative for all objects $X_1$ and $X_2$ in $\C$. 
  
\end{itemize}
\end{corollary}

\subsection{Categories where coproduct inclusions are extensive}
A \emph{coproduct inclusion} in a category $\C$ is a morphism $\iota \: A \to X$ in $\C$ such that there exists a morphism $\iota'\: A' \to X$ such that 
$$\xymatrix{A \ar[r]^{\iota} & X   & A' \ar[l]_{\iota'}}$$
is a coproduct diagram.

\begin{definition}
  Let $\C$ be a category with finite coproducts which admits all pullbacks along coproduct inclusions. Coproducts are said to be \emph{disjoint} in $\C$ if for any coproduct diagram $X_1 \xrightarrow{\iota_1} X \xleftarrow{\iota_2} X_2$,
  the squares 
  \[
  \xymatrix{
  X_1 \ar[d]_{1_{X_1}} \ar[r]^{1_{X_1}} & X_1 \ar[d]_{\iota_1} & 0 \ar[d] \ar[l]\\ 
  X_1 \ar[r]_{\iota_1} & X & \ar[l]^{\iota_2} X_2}
  \]
  are pullbacks. Note, the left-hand square being a pullback implies that coproduct inclusions are monomorphisms.  
\end{definition}

\begin{proposition}\label{proposition: coproduct inclusions are regular monos}
    In a category $\C$ with finite coproducts where every coproduct inclusion satisfies $(\ExtSecond)$, coproducts are disjoint and  every coproduct inclusion is a regular monomorphism. 
\end{proposition} 
\begin{proof}
It follows immediately that coproducts are disjoint. Consider the coproduct inclusion $X \xrightarrow{\iota} A$, with complementary inclusion $Y \xrightarrow{\iota'} A$. Form also the coproduct diagram $A \xrightarrow{i} A+Y \xleftarrow{i'} Y$. Then, the rows in the following diagram are coproducts
\[\xymatrix{
X \ar[r]^-{\iota} \ar[d]_{\iota}& A \ar[d]^{\iota_1+1_Y} & Y \ar[d]^{1_Y} \ar[l]_-{\iota_2}\\
A \ar[r]_-{i} & A+Y & Y \ar[l]^-{i'}
}\]
so that both squares are pullbacks. Hence,  $\iota$ is an equaliser of $i$ and $\iota_1 + 1_Y$.
\end{proof} 

\begin{proposition}[Coproduct complements are unique]
  Let $\C$ be a category with an initial object with disjoint coproducts. Then coproduct complements are unique, i.e., if
  \begin{align}
    \label{diagram: two coproduct diagrams}
    \xymatrix{X_1 \ar[r]^{\iota_1} & X & X_2 \ar[l]_{\iota_2}, &  X_1 
    \ar[r]^{\iota_1} & X & X_2' \ar[l]_{\iota_2'}}
  \end{align}
  are both coproduct diagrams in $\C$, then there is an isomorphism $\sigma\:X_2 \to X_2'$ such that  $\iota_2' = \sigma \iota_2$.
\end{proposition}
\begin{proof}
  Suppose we have two coproduct diagrams as in (\ref{diagram: two coproduct diagrams}).
  Then, each right-hand square in
  \[\xymatrix{
    0 \ar[r] \ar[d] \pbl& X_2 \ar[d]_{\iota_2} & X_2 \pbr \ar[l] \ar[d] & 0 \ar[r] \ar[d] \pbl & X_2' \ar[d]_{\iota_2'} & X_2' \pbr \ar[l] \ar[d] \\
    X_1 \ar[r]_{\iota_1} & X & \ar[l]^{\iota_2'} X_2' & X_1 \ar[r]_{\iota_1} & X & \ar[l]^{\iota_2} X_2}\]
  
  is a pullback. Thus, there is an isomorphism $\sigma: X_2 \to X_2'$ such that $\iota_2 = \iota_2' \sigma$. 
\end{proof}

\begin{proposition}
  \label{prop: crisp implies extensive}
  In a category with an initial object and disjoint coproducts, if all coproduct inclusions satisfy $(\ExtFirst)$ then any morphism satisfying $(\ExtFirst)$ is extensive.
\end{proposition}

\begin{proof}
  It suffices to show that every identity morphism in $\C$ satisfies $(\ExtSecond)$ by Corollary~\ref{corollary: Ext1 are extensive}. 
  Suppose that we are given any diagram
\begin{equation}
  \label{diagram: coprod inclusions crisp and disjoint}
  \xymatrix{
    A_1 \ar[r]^{a_1} \ar[d]_{i_1} & A \ar[d]_{1_A} & A_2 \ar[l]_{a_2} \ar[d]^{i_2} \\
    B_1 \ar[r]_{b_1} & A & B_2 \ar[l]^{b_2}
  }
\end{equation}
  where the top and bottom rows are coproduct diagrams. Note that since $a_1$ and $a_2$ are monomorphisms, it follows that $i_1$ and $i_2$ are monomorphisms. Consider the commutative diagram diagram below.
  \[
\xymatrix{
      A_1 \ar[r]^{i_1} \ar[d]_{1_{A_1}} & B_1 \ar[d]^{b_1} & 0 \ar[l] \ar[d] & 0 \ar[l] \ar[d] \\
      A_1 \ar[r]_{a_1} & A & B_2 \ar[l]^{b_2} & A_2 \ar[l]^{i_2}
}
\]
  The left square is a pullback (since $b_1$ is a monomorphism). Further, the right square is a pullback since $i_2$ is a monomorphism, and the middle square is a pullback since coproducts are disjoint. Therefore $i_1$ is an isomorphism. We can likewise show that $i_2$ is an isomorphism, so that the two squares in (\ref{diagram: coprod inclusions crisp and disjoint}) are pullbacks.
\end{proof}

\begin{corollary} \label{corollary: disjoint E1 implies extensive}
The following are equivalent for a category $\C$ with an initial object. 
\begin{enumerate}
    \item Every coproduct inclusion in $\C$ is extensive. 
    \item Coproducts are disjoint, and every coproduct inclusion satisfies $(\ExtFirst)$. 
\end{enumerate}
\end{corollary}

\begin{proposition}
Let $\C$ be a category where every coproduct inclusion is extensive. Then the pullback of an extensive morphism along a coproduct inclusion exists and is extensive. 
\end{proposition}
\begin{proof}
Let $f:A \to B$ be an extensive morphism in $\C$, then $f$ admits pullbacks along coproduct inclusions by virtue of $(\ExtFirst)$. Now consider the diagram 
\[
\xymatrix{
A_1 \ar[r]^{a_1} \ar[d]_{f_1} & A \ar[d]^f & A_2 \ar[l]_{a_2} \ar[d]^{f_2} \\ 
B_1 \ar[r]_{b_1} & B & B_2 \ar[l]^{b_2}
}
\]
where the squares are pullbacks. Since $f$ is extensive, the top row is a coproduct and hence $a_1$ is extensive. Then by Proposition~\ref{proposition: composite of extensive} it follows that $fa_1$ is extensive, and hence $b_1 f_1$ is extensive. Since $b_1$ is a coproduct inclusion, by Proposition~\ref{proposition: if extensive implies f extensive} it follows that $f_1$ is extensive. 
\end{proof}

\subsection{Coextensivity of product projections} \label{section: coextensivity of product projections}
The \emph{strict refinement property} was initially introduced in  \cite{ChangJonssonTarski1964}, which has a straightforward generalisation to categories given by the following
\begin{definition} \label{definition: strict refinement property}
	An object $X$ in a category $\C$ is said to have the (finite) \emph{strict refinement property} if for any two (finite) product diagrams $(X \xrightarrow{a_i} A_i)_{i \in I}$ and $(X \xrightarrow{b_j} B_j)_{j \in J}$, there exist families of morphisms $(A_i \xrightarrow{\alpha_{i,j}} C_{i,j})_{i\in I,j\in J}$ and $(B_j \xrightarrow{\beta_{i,j}} C_{i,j})_{i\in I, j\in J}$ such that $\alpha_{i,j}a_i =\beta_{i,j}b_j$ and the diagrams $(A_i \xrightarrow{\alpha_{i,s}} C_{i,s})_{s\in J}$ and $(B_j \xrightarrow{\beta_{t,j}} C_{t,j})_{t \in I}$ are product diagrams for any $i\in I$ and $j\in J$. The category $\C$ satisfies the strict refinement property if every object in $\C$ does. 
\end{definition}
The simplest non-trivial strict refinement for an object $X$ is for binary product diagrams. 
\begin{definition} \label{definition: binary strict refinement property}
 An object $X$ in a category $\C$ is said to satisfy the \emph{binary strict refinement property} if for an object $X$ given two binary product diagrams 
\[
A_1 \xleftarrow{a_1} X \xrightarrow{a_2} A_2, \quad  B_1 \xleftarrow{b_1} X \xrightarrow{b_2} B_2, 
\]
there exists a commutative diagram 
\[
\xymatrix{
C_{1,1} & A_1 \ar[r]^{\alpha_{1,2}} \ar[l]_{\alpha_{1,1}} & C_{1,2}\\
B_1 \ar[u]^{\beta_{1,1}} \ar[d]_{\beta_{2,1}}  & \ar[l]_{b_1} X \ar[d]^{a_2} \ar[u]_{a_1} \ar[r]^{b_2} & B_2 \ar[u]_{\beta_{1,2}} \ar[d]^{\beta_{2,2}} \\
C_{2,1} & A_2 \ar[l]^{\alpha_{2,1}}\ar[r]_{\alpha_{2,2}} & C_{2,2}
}
\]
were each edge is a binary product diagram. The category $\C$ then satisfies the binary strict refinement property if every object in $\C$ does.   
\end{definition}

 The proof of the lemma below (which appears as Proposition~2.4 in \cite{hoefnagel2019products}) is standard and left to the reader.
\begin{lemma} \label{lemma: common-coequaliser}
Given any reasonably commutative diagram  
\[
\xymatrix{
 C_1 \ar@{->>}[d]_e \ar@<-.5ex>[r]_{v_1} \ar@<.5ex>[r]^{u_1} & X_1 \ar[d]^f \ar[r]^{q_1} & Q_1 \ar[d]^g \\
 C_2 \ar@<-.5ex>[r]_{v_2} \ar@<.5ex>[r]^{u_2} & X_2 \ar[r]_{q_2} & Q_2
}
\]
in any category, where the top row is a coequaliser diagram and $e$ is an epimorphism, the right-hand square is a pushout if and only if the bottom row is a coequaliser. 
\end{lemma}
Note that in every category with finite products, every product projection admits a kernel pair. Specifically, for every product diagram $\xymatrix{X_1 & X_1 \times X_2 \ar[r]^-{p_2} \ar[l]_-{p_1} & X_2}$, the morphism $p_1$ has kernel pair
\[
\xymatrix{
X_1 \times (X_2 \times X_2)  \ar@<0.5ex>[r]^-{1_{X_1} \times \pi_1} \ar@<-0.5ex>[r]_-{1_{X_1} \times \pi_2} & X_1 \times X_2
}
\]
where $\pi_1$ and $\pi_2$ are the first and second projection morphisms of the product $X_2 \times X_2$ respectively.

\begin{lemma} \label{lemma: codisjoint products}
Let $\C$ be a category with finite products where every product projection in $\C$ is a regular epimorphism. Then binary products in $\C$ are co-disjoint, i.e., coproducts are disjoint in $\C^{\op}$. 
\end{lemma}
\begin{proof}
For any two objects $X,Y$ in $\C$ we may apply Lemma~\ref{lemma: common-coequaliser} to  the diagram 
\[
\xymatrix{
 X \times (Y \times Y) \ar@{->>}[d] \ar@<-.5ex>[r]_-{1_X \times \pi_1} \ar@<.5ex>[r]^-{1_X \times \pi_2} & X \times Y \ar[d] \ar[r] & X \ar[d] \\
 Y \times Y \ar@<-.5ex>[r]_{\pi_1} \ar@<.5ex>[r]^{\pi_2} & Y \ar[r] & 1
}
\]
to get that the right square is a pushout. 
\end{proof}

\begin{proposition} \label{proposition: coextensive product projections = strict refinement property}
Let $\C$ be a category with finite products where every product projection is a regular epimorphism. For any object $X$ in $\C$, each product projection of $X$ is coextensive if and only if $X$ satisfies the binary strict refinement property.
\end{proposition}

\begin{proof}
Suppose first that each product projection of $X$ is coextensive. Let $X$ have the following two product diagrams:
\[
A_1 \xleftarrow{a_1} X \xrightarrow{a_2} A_2, \quad B_1 \xleftarrow{b_1} X \xrightarrow{b_2} B_2.
\]
By coextensivity of these morphisms, we can form pushouts of each $a_i$ along each $b_j$, obtaining the binary product diagrams along the edges of the outer square in the refinement diagram.

Conversely, suppose $X$ satisfies the binary strict refinement property. By Lemma~\ref{lemma: codisjoint products}, since every product projection in $\C$ is a regular epimorphism, products are co-disjoint. Thus, by the dual of Corollary~\ref{corollary: disjoint E1 implies extensive}, we need only show that each product projection satisfies $(\ExtFirstop)$. Let $\pi: X \to A$ be any product projection, and consider a product diagram
\[
X_1 \xleftarrow{x_1} X \xrightarrow{x_2} X_2.
\]
Using the binary strict refinement property, we construct the bottom two squares in the  diagram
\[
\xymatrix{
K_1 \ar@{.>}@<0.5ex>[d]^-{u_2} \ar@{.>}@<-0.5ex>[d]_-{u_1} & \ar@{->>}[l]_{p_1} K \ar@{->>}[r]^{p_2} \ar@<0.5ex>[d]^-{k_2} \ar@<-0.5ex>[d]_-{k_1} & K_2 \ar@{.>}@<0.5ex>[d]^-{v_2} \ar@{.>}@<-0.5ex>[d]_-{v_1} \\
X_1 \ar@{->>}[d]_{\pi_1} & X \ar@{->>}[l]^-{x_1} \ar@{->>}[d]_{\pi} \ar@{->>}[r]_-{x_2} & X_2 \ar@{->>}[d]^{\pi_2} \\
A_1 & \ar[l]^-{y_1} A \ar[r]_-{y_2} & A_2
}
\]
where $\pi_1$ and $\pi_2$ are product projections and the bottom row is a product diagram. The parallel pairs in the top squares are constructed by taking the kernel pairs of $\pi_1$, $\pi$, and $\pi_2$. Hence, each column in this diagram forms a coequaliser diagram. Furthermore, the top row is a product digram (since the bottom row is), so that $p_1$ and $p_2$ are regualr epimorphisms. 
Therefore, by Lemma~\ref{lemma: common-coequaliser}, the bottom two squares are pushouts. Thus, $\pi$ satisfies $(\ExtFirstop)$, proving it is coextensive.
\end{proof}

The following theorem is an adaptation taken from \cite{Hoefnagel2020}; we include its proof for the sake of completeness. 

\begin{theorem} \label{proposition: projection coextensive implies strict refinement}
Let $\C$ be a category with (finite) products and let $X$ be an object with coextensive product projections. Then, $X$ has the (finite) strict refinement property. 
\end{theorem}
\begin{proof}
Suppose that  $(X \xrightarrow{a_i} A_i)_{i \in I}$ and $(X \xrightarrow{b_j} B_j)_{j \in J}$ are any two (finite) product diagrams for $X$. Let $\overline{A_n}$ be the product of the $A_i$'s where $i \neq n$ and let $\overline{a_n}:X \rightarrow \overline{A_n}$ be the induced morphism $(a_i)_{i\neq n}$, and similarly let $\overline{B_m}$ be the product of the $B_j$'s where $j \neq m$. For each $n \in I$ and $m \in J$ there is a diagram
\[
\xymatrix{
	A_n \ar[d]_-{\alpha_{n,m}} &  X\ar[r]^{\overline{a_n}} \ar[d]^{b_m} \ar[l]_{a_n} & \overline{A_n}  \ar[d]^-{\overline{\alpha_{n,m}}}\\
	C_{n,m}  & \ar[r]_-{\beta'_{n,m}} B_m \ar[l]^-{\beta_{n,m}} & \overline{C_{n,m}}
}
\]
where each square is a pushout, and the bottom row is a product diagram, since $b_n$ is a product projection of $X$. In the diagram 
\[
\xymatrix{
	A_n \ar[d]_-{(\alpha_{n,j})_{j \in J}} & &  X\ar[rr]^{\overline{a_n}} \ar[d]^{(b_j)_{j \in J}} \ar[ll]_{a_n} & & \overline{A_n}  \ar[d]^-{(\overline{\alpha_{n,j}})_{j \in J}} \\
	\prod \limits_{j \in J} C_{n,j}  & & \ar[rr]_-{\prod \limits_{j \in J}\beta'_{n,j}} \prod \limits_{j \in J}B_j \ar[ll]^-{\prod \limits_{j \in J}\beta_{n,j}} & & \prod\limits_{j\in J} \overline{C_{n,j}}
}
\]
the bottom row is a product diagram, and hence each square is a pushout. Since the central vertical morphism in the diagram is an isomorphism, it follows that the morphism $(\alpha_{n,j})_{j\in J}$ is an isomorphism, and we can similarly obtain $(\beta_{i, m})_{i \in I}$ as an isomorphism. 
\end{proof}
As a result of Proposition~\ref{proposition: projection coextensive implies strict refinement} and Proposition~\ref{proposition: coextensive product projections = strict refinement property}, we have the following.

\begin{corollary} 
Let $X$ be an object in a category $\C$ with (finite) products, where every product projection in $\C$ is a regular epimorphism. Then $X$ has the (finite) strict refinement property if and only if every product projection of $X$ is coextensive. 
\end{corollary}

\subsubsection{Examples of coextensive product projections}
As a result of Hashimoto's theorem \cite{Hashimoto} for partially ordered sets, every non-empty connected partially ordered set satisfies the strict refinement property. Let $\Pos$ denote the category of non-empty partially ordered sets. It is readily seen that every product projection in $\Pos$ is a regular epimorphism. Thus, we have the following result:

\begin{proposition}
Every product projection of a connected partially ordered set is coextensive in $\Pos$. 
\end{proposition}

\begin{remark}
Consider the full subcategory $\CPos$ of connected partially ordered sets. Note that $\CPos$ has finite products, but does not have equalisers. Moreover $\CPos$ has all pushouts, and they are computed as in $\Pos$. Consequently, every product projection in $\CPos$ is coextensive.
\end{remark}
Since product projections preserve meets and joins, another consequence of Hashimoto's theorem is that every non-empty semi-lattice satisfies the strict refinement property in the category $\mathsf{SLat}$ of semi-lattices.

\begin{proposition}
    Every product projection in the category $\mathsf{SLat}$ of semi-lattices is coextensive. 
\end{proposition}
Every monoid $M$ admits a \emph{center} $\mathsf{Z}(M)$, which is given by 
\[
\mathsf{Z}(M) = \{x \in M \mid \forall y \in M [xy = yx]\}. 
\]
Then $M$ is said to be \emph{centerless} if $\mathsf{Z}(M) = \{0\}$. As a consequence of \cite{AlgebrasLatticesVarietiesI} (see Corollary 2 in Section 5.6), every centerless monoid has the strict refinement property in $\Mon$. For this reason, we have the following result:
\begin{proposition}
Every product projection of a centerless monoid is coextensive in $\Mon$. 
\end{proposition}

As shown in \cite{Hoefnagel2020}, in a Barr-exact \cite{BarrGrilletOsdol1971} Mal'tsev category, every centerless object has coextensive product projections. 

\subsection{$\M$-coextensivity}  
The main purpose of the paper \cite{Hoefnagel2020} was to study categorical aspects of the \emph{strict refinement property} \cite{ChangJonssonTarski1964} for varieties of universal algebras. The approach taken in that paper is object-wise, through a notion introduced there of an $\mathcal{M}$-coextensive object. Below, we describe the dual notion. 

\begin{definition} \label{def-M-pushout}
Let $\C$ be a category and $\M$ a class of morphisms from $\C$. A commutative square
\[
\xymatrix{
P \ar[r]^a \ar[d]_b & A \ar[d] \\
B \ar[r] & X
}
\] 
in $\C$ is called an \emph{$\M$-pullback} if it is a pullback in $\C$, and both $a$ and $b$ are morphisms in $\M$.
\end{definition}
 \begin{definition} \label{def-M-coextensive}
 	Let $\C$ be a category and $\M$ a class of morphisms in $\C$. An object $A$ is said to be $\M$-\emph{extensive} if every morphism in $\M$ with codomain $A$ admits an $\M$-pullback along every coproduct inclusion of $A$, and in each commutative diagram 
 	\[
 	\xymatrix{
 		X_1 \ar[d] \ar[r] &  X \ar[d]  & X_2 \ar[l] \ar[d] \\
 		A_1 \ar[r] &  A   & A_2 \ar[l]
 	}
 	\]
  where the bottom row is a coproduct diagram and the vertical morphisms belong to $\M$, the top row is a coproduct diagram if and only if both squares are $\M$-pullbacks. 
 \end{definition}
  \begin{remark}
The notion above is the dual notion of $\M$-coextensive object given in \cite{Hoefnagel2020}. That is, given a class of morphisms $\M$ in a category $\C$ with finite coproducts, an object $A$ is $\M$-coextensive if and only if $A$ is $\M$-extensive in $\C^{\op}$. 
  \end{remark}
\begin{proposition}
Let $\mathcal{M}$ be a class of morphisms in a category $\C$, where $\M$ is stable under pullbacks along coproduct inclusions. Then an object $A$ is $\M$-extensive if and only if every morphism in $\M$ with codomain $A$ is extensive. 
\end{proposition}
\begin{remark}
In any category $\C$ we may consider $\mathcal{M}$ to be the class of monomorphisms in $\C$, in which case $\mathcal{M}$ is closed under pullbacks. Then as defined in \cite{DvorakZemlicka2022}, an object is called mono-extensive if it is $\mathcal{M}$-extensive. Moreover, what is termed a unique-decomposition category (UD-category) has the property that every object in such a category is mono-extensive. For this reason, in every UD-category every monomorphism is extensive. 
\end{remark}
\begin{corollary}
Let $\C$ be a category with finite products in which every product projection is coextensive, then every object in $\C$ is $\mathcal{M}$-coextensive where $\mathcal{M}$ is the class of all product projections in $\C$. 
\end{corollary}

The notion of \emph{Boolean category} introduced in \cite{PredicateTransformer} is equivalent to the following. 
\begin{definition} \label{def:Boolean-category}
A category $\C$ with finite coproducts is \emph{Boolean} if it satisfies the following: 
	\begin{enumerate}
		\item $\C$ admits all pullbacks along coproduct inclusions, and the class of coproduct inclusions is pullback stable;
		\item  Every coproduct inclusion satisfies $(\ExtFirst)$; 
		\item If $Y \xrightarrow{t} X \xleftarrow{t} Y$ is a coproduct diagram, then $X$ is an initial object.
	\end{enumerate}
\end{definition}
All else being equal, the only difference between a Boolean category and a category with finite coproducts in which every coproduct inclusion is coextensive is that coproduct inclusions need not be pullback stable.

\subsection{Commutativity of finite products with coequalisers}
Recall that for a category with binary products, we say that binary products commute with coequalisers in $\C$ if for any two coequaliser diagrams
\[
    \xymatrix{
    C_1 \ar@<-.5ex>[r]_{v_1} \ar@<.5ex>[r]^{u_1} & X_1 \ar[r]^{q_1} & Q_1  & C_2 \ar@<-.5ex>[r]_{v_2} \ar@<.5ex>[r]^{u_2}& X_2 \ar[r]^{q_2} & Q_2,
    }
\]
in $\C$, the diagram
\[
\xymatrix{
& C_1 \times C_2 \ar@<-.5ex>[r]_-{v_1 \times v_2} \ar@<.5ex>[r]^-{u_1 \times u_2}  & X_1 \times X_2 \ar[r]^-{q_1 \times q_2} & Q_1 \times Q_2, \\
}
\]
is a coequaliser diagram.
In \cite{Hoefnagel2019a} the commutativity of finite products with coequalisers was considered, and shown to hold in any coextensive category. We remark here that this property is intimately connected to the topic of Huq-centrality of morphisms (see \cite{Hoefnagel2023}). First, we deal with the dual property. 

\begin{proposition} \label{proposition: sum of equalisers is an equaliser}
Let $\C$ be a category with coproducts and equalisers, where coproduct inclusions are monic, and where every regular monomorphism in $\C$ is extensive. Then equalisers commute with finite coproducts in $\C$, i.e.\ for any two equaliser diagrams
\[
    \xymatrix{
    E_1 \ar[r]^{e_1}& X_1 \ar@<-.5ex>[r]_{f_1} \ar@<.5ex>[r]^{g_1} & Y_1   & E_2 \ar[r]^{e_2}&  X_2 \ar@<-.5ex>[r]_{f_2} \ar@<.5ex>[r]^{g_2}& Y_2 
    }
\]
in $\C$, the diagram 
\[
\xymatrix{ 
E_1 + E_2 \ar[r]^{e_1 + e_2} & X_1 + X_2 \ar@<-.5ex>[r]_-{f_1 + f_2} \ar@<.5ex>[r]^-{g_1 + g_2}  & Y_1 + Y_2  \\ 
}
\]
is an equaliser diagram in $\C$.
\end{proposition}
\begin{proof}
Given the two equaliser diagrams above, consider the equaliser $e: E \rightarrow X_1 + X_2$ of $f_1 + f_2$ and $g_1 + g_2$. Then, in the diagram 
\[
\xymatrix{
E_1 \ar@{..>}[r] \ar[d]_{e_1} & E \ar[d]^e & E_2   \ar@{..>}[l] \ar[d]^{e_2} \\ 
X_1 \ar[r] \ar@<-.5ex>[d]_-{f_1} \ar@<.5ex>[d]^-{g_1} & X_1 + X_2 \ar@<-.5ex>[d]_-{f_1 + f_2} \ar@<.5ex>[d]^-{g_1 + g_2} & X_2 \ar[l] \ar@<-.5ex>[d]_-{f_2} \ar@<.5ex>[d]^-{g_2} \\ 
Y_1 \ar[r] & Y_1 + Y_2 & Y_2 \ar[l] 
}
\]
where the dotted arrows are induced by the equaliser $e$, the upper squares are pullbacks by the dual of Lemma~\ref{lemma: common-coequaliser}, and hence the top row is a coproduct. 
\end{proof}

By the proposition above, we note that in a finitely complete category with coequalisers in which every regular epimorphism is coextensive, we have that finite products commute with coequalisers. The proposition below asserts a converse to this fact. 

\begin{proposition}
Let $\C$ be a finitely complete category that admits pushouts of regular epimorphisms along product projections, and suppose that every terminal morphism $X \to 1$ in $\C$ is a regular epimorphism. If finite products commute with coequalisers in $\C$ and every split monomorphism in $\C$ is coextensive, then $\C$ is coextensive. 
\end{proposition}

\begin{proof}
Suppose that binary products commute with coequalisers in $\C$ and that every split monomorphism is coextensive. Note that regular epimorphisms are stable under binary products in $\C$, so that every product projection $X\times Y \to X$ is a regular epimorphism, since every product projection is a product of a terminal morphism and an identity morphism. It suffices to show that every product projection in $\C$ satisfies $(\ExtFirstop)$ (by Proposition~\ref{corollary: Ext1 are extensive} and the dual of Corrolary~\ref{corollary: C extensive iff split epis and inclusions are}). Let $X_1 \xleftarrow{\pi_1} X \xrightarrow{\pi_2} X_2$
 be any product diagram, and suppose that $q:X \to Y$ is any product projection of $X$. Consider the diagram
\[\xymatrix{
X_1 \ar@/_2pc/[dd]_{1_{X_1}} \ar[d]_{f_1}  & \ar@{->>}[l]_{\pi_1} X  \ar@{->>}[r]^{\pi_2} \ar@{>->}[]+<0ex,-2.2ex>;[d]_-{\theta} & X_2 \ar[d]^{f_2} \ar@/^2pc/[dd]^{1_{X_2}}
\\
P_1 \pol \ar@{.>}@<0.5ex>[d]^-{u_2} \ar@{.>}@<-0.5ex>[d]_-{u_1} & \ar@{->>}[l]^{p_1} K \ar@{->>}[r]_{p_2} \ar@<0.5ex>[d]^-{k_2} \ar@<-0.5ex>[d]_-{k_1} & P_2 \por \ar@{.>}@<0.5ex>[d]^-{v_2} \ar@{.>}@<-0.5ex>[d]_-{v_1} \\
X_1 \ar@{->>}[d]_{q_1} & X \ar@{->>}[l]^-{\pi_1} \ar@{->>}[d]_{q} \ar@{->>}[r]_-{\pi_2}& X_2 \ar@{->>}[d]^{q_2} \\
Y_1 & \ar[l]^-{\alpha_1} Y \ar[r]_-{\alpha_2} & Y_2}
\]
where $(K,k_1,k_2)$ is the kernel pair of $q$, $\theta$ is the diagonal inclusion and the bottom two squares are the pushouts of $q$ along $\pi_1$ and $\pi_2$. Since the morphism $\theta$ is split monomorphism, we may push out $\theta$ along $\pi_1$ and $\pi_2$, producing a product diagram for $K$, as well the dotted arrows making the middle squares reasonably commute. Then, since $p_1$ and $p_2$ are (regular) epimorphisms, it follows by Lemma~\ref{lemma: common-coequaliser} that $q_1$ is a coequaliser of $u_1$ and $u_2$, and $q_2$ is a coequaliser of $v_1$ and $v_2$. Then, since binary products commute with coequalisers, it follows that $q_1 \times q_2$ is a coequaliser of $k_1$ and $k_2$, so that $q$ is isomorphic to $q_1 \times q_2$ and hence the bottom row is a product diagram. 
\end{proof}

\section{Coextensivity of morphisms in regular categories} 
Recall that a morphism $f\:X \to Y$ in a category $\C$ is called a \emph{regular epimorphism} in $\C$ if $f$ is the coequaliser of some parallel pair of morphisms. Recall that  category $\C$ is \emph{regular} \cite{BarrGrilletOsdol1971} if it has finite limits, coequalisers of kernel pairs, and if the pullback of a regular epimorphism along any morphism is again a regular epimorphism. Listed below are some elementary facts about morphisms in a regular category $\C$: 
\begin{itemize}
    \item Every morphism in $\C$ factors as a regular epimorphism followed by a monomorphism. 
    \item Regular epimorphisms in $\C$ are stable under finite products. 
    \item Every extremal epimorphism in $\C$ is a regular epimorphism in $\C$. 
\end{itemize}

Given any object $X$ in a regular category $\C$, consider the preorder of all monomorphisms with codomain $X$. The posetal-reflection of this preorder is $\Sub(X)$ -- the poset of \emph{subobjects} of $X$. For any morphism $f: X \rightarrow Y$ in $\C$ there is an induced Galois connection 
\[
\xymatrix{
\Sub(X) \ar@/^1pc/[r] \ar@{}[r]|{\bot} &  \Sub(Y) \ar@/^1pc/[l]
}
\]
given by direct image and inverse image. This is defined as follows: given a subobject $A$ of $X$ represented by a monomorphism $A_0 \xrightarrow{a} X$, the direct image $f(A)$ of $A$ along $f$ is defined to be the subobject represented by the monomorphism part of a (regular epimorphism, monomorphism)-factorisation of $f   a$. Given a subobject $B$ of $Y$ represented by a monomorphism $B_0 \xrightarrow{b} Y$, we define the inverse image $f^{-1}(B)$ of $B$ to be the subobject of $X$ represented by the monomorphism obtained from pulling back $b$ along $f$. A \emph{relation} $R$ from $X$ to $Y$ is a subobject of $X \times Y$, i.e., an isomorphism class of monomorphisms with codomain $X \times Y$. Given such a relation $R$ represented by a monomorphism $(r_1, r_2)\: R_0 \to X \times Y$, we define the \emph{opposite} relation $R^{\op}$ to be the relation represented by the monomorphism $(r_2, r_1)\: R_0 \to Y \times X$. Given an object $X$ in a category, we write $\Delta_X$ for the \emph{diagonal relation}, that is the relation represented by $(1_X,1_X)\: X \to X \times X$, and we write $\nabla_X$ for the relation on $X$ represented by the identity morphism on $X \times X$. Given a morphism $f\:X \to Y$ in $\C$, the kernel pair of $f$ represents an equivalence relation $\Eq(f)$ --- the kernel relation of $f$. 

Regular categories possess a well-behaved composition of relations, which is defined as follows: given a relation $R$ from $X$ to $Y$ and a relation $S$ from $Y$ to $Z$, and two representatives $(r_1, r_2):R_0 \rightarrow X \times Y$ and $(s_1,s_2):S_0 \rightarrow Y \times Z$ of $R$ and $S$ respectively, form the pullback of $s_1$ along $r_2$ as below
\[
\xymatrix{
P \ar[r]^{p_2} \ar[d]_{p_1} & S_0 \ar[d]^{s_1} \\
R_0 \ar[r]_{r_2} & Y
}
\]
so that $R \circ S$ is defined to be the relation represented by the monomorphism part of any regular-image factorisation of $(r_1 p_1, s_2p_2):P \rightarrow X \times Z$. This composition of relations is then associative by virtue of the category $\C$ being regular. A relation $R$ on an object $X$ is called:
\begin{itemize}
    \item \emph{reflexive} whenever $\Delta_X \leqslant R$;
    \item \emph{symmetric} whenever $R^{\op} \leqslant R$; 
    \item \emph{transitive} whenever $R \circ R \leqslant R$; an
    \item \emph{equivalence} if $R$ is reflexive, symmetric and transitive.
\end{itemize}
\begin{notation} \label{notation: kernel equivalence relation}
Given any morphism $f\:X \to Y$ in any category, the kernel pair $(K, k_1, k_2)$ of $f$ is an equivalence relation, which we will denote by $\Eq(f)$ in what follows. 
\end{notation}
An equivalence relation $E$ on an object $X$ is called \emph{effective} if it is the kernel equivalence relation of some morphism, i.e.\ there is an $f\:X \to Y$ such that $E = \Eq(f)$. 
\begin{definition} \label{definition: Barr-exact category}
A regular category $\C$ is said to be \emph{Barr-exact} if every equivalence relation is effective. 
\end{definition}
    Finally, given a relation $R$ on an object $X$ represented by $(R_0,r_1,r_2)$ and a relation $S$ on an object $Y$ represented by $(S_0,s_1,s_2)$ we define their \emph{product} to be the relation $R \times S$ on  $X \times Y$ to be the relation represented by $(R_0 \times S_0, r_1 \times s_1, r_2 \times s_2)$. Below, we list properties relevant to this section that hold in any regular category, as derived from the calculus of relations. The proofs can be found in \cite{TheartMSc}. 
\begin{itemize}
    \item For any relation $R$ from $X$ to $Y$ in $\C$, we have $\Delta_Y \circ R = R = R \circ \Delta_X$ for any relation $R$ from $X$ to $Y$ (Proposition~44 in \cite{TheartMSc});
    \item $\nabla_X \circ R = \nabla_X = R \circ \nabla_X$, for any reflexive relation $R$ on $X$ (Proposition 45 in \cite{TheartMSc});
    \item $f(R \circ S) \leqslant f(R) \circ f(S)$ for any morphism $f: X \rightarrow Y$ and relations $R$ and $S$ on $X$ (Proposition 88 in \cite{TheartMSc});
    \item A reflexive relation $R$ is transitive if and only if $R = R \circ R$ (Corollary~48 in \cite{TheartMSc});
    \item $(R\circ R') \times (S \circ S') = (R\times S) \circ (R' \times S')$ for any relations $S$ and $R$ such that the composition makes sense (Proposition~40 in \cite{TheartMSc}); 
    \item $f(f^{-1}(R)) = R$ for any relation $R$ on $X$ and any regular epimorphism $f: X \rightarrow Y$ (Proposition 87 in \cite{TheartMSc}); 
    \item  $f^{-1}(f(R)) = \mathrm{Eq}(f) \circ R \circ \mathrm{Eq}(f)$ for any relation $R$ on $X$ where $f$ is a regular epimorphism (Proposition~93 in \cite{TheartMSc});
    \item  $f(f^{-1}(R) \circ f^{-1}(S)) = R \circ S$ for any regular epimorphism $f$ and any relations $R,S$ on $Y$ (Proposition~92 in \cite{TheartMSc}).
\end{itemize}
The proof of the proposition below follows from a routine calculation, which is demonstrated in Proposition 91 of \cite{TheartMSc}.
\begin{proposition} \label{proposition: kernel pair of projections}    In any category with binary coproducts, the kernel pairs of the projections in any product diagram
$X_1 \xleftarrow{\pi_1} X \xrightarrow{\pi_2} X_2$
are given by the relations $\mathrm{Eq}(\pi_1) = \Delta_{X_1} \times \nabla_{X_2}$ and $\mathrm{Eq}(\pi_2) = \nabla_{X_1} \times \Delta_{X_2}$. 
\end{proposition}

\begin{lemma}
  \label{lemma: equivalence relations under reg epi}
  Let $\C$ be a regular category and $E$  an equivalence relation on $X$. Then for any product diagram
 $X_1 \xleftarrow{\pi_1} X \xrightarrow{\pi_2} X_2$
  where $\pi_1$ and $\pi_2$ are regular epimorphisms, $E = \pi_1(E) \times \pi_2(E)$ implies that $\pi_1(E)$ and $\pi_1(E)$ are equivalence relations.
\end{lemma}
\begin{proof}
  It is well-known that reflexivity of relations are preserved by any morphism, and symmetry is preserved by regular epimorphisms. Since $\pi_1$ and $\pi_2$ are regular epimorphisms, it suffices to show that $\pi_1(E)$ and $\pi_2(E)$ are transitive. Note that
  \begin{align*}
    \mathrm{Eq}(\pi_1) \circ E &= (\Delta_{X_1}\times \nabla_{X_2}) \circ (\pi_1(E) \times \pi_2(E))  \tag{by Proposition~\ref{proposition: kernel pair of projections}} \\
    &= (\Delta_{X_1} \circ \pi_1(E)) \times (\nabla_{X_2} \circ \pi_2(E)) \\
    &=(\pi_1(E) \circ\Delta_{X_1}) \times (\pi_2(E) \circ \nabla_{X_2}) \\
    &= (\pi_1(E) \times \pi_2(E)) \circ (\Delta_{X_1} \times \nabla_{X_2}) \\
    &= E \circ \mathrm{Eq}(\pi_1).
  \end{align*}
  Hence,
  \begin{align*}
    \pi_1(E) \circ \pi_1(E) &= \pi_1(\pi_1^{-1}(\pi_1(E)) \circ \pi_1^{-1}(\pi_1(E))) \\
    &=\pi_1(\mathrm{Eq}(\pi_1) \circ E \circ \mathrm{Eq}(\pi_1)\circ\mathrm{Eq}(\pi_1) \circ E \circ \mathrm{Eq}(\pi_1)) \\
    &=\pi_1(\mathrm{Eq}(\pi_1) \circ E \circ\mathrm{Eq}(\pi_1) \circ E \circ \mathrm{Eq}(\pi_1))  \\
    &= \pi_1(\mathrm{Eq}(\pi_1) \circ E \circ E  \circ\mathrm{Eq}(\pi_1) \circ \mathrm{Eq}(\pi_1)) \\
    &=\pi_1(\mathrm{Eq}(\pi_1) \circ E \circ \mathrm{Eq}(\pi_1))  \\
    &=\pi_1(\pi_1^{-1}(\pi_1(E)))\\
    &=\pi_1(E).
  \end{align*}
  Therefore, $\pi_1 (E)$ is transitive, and hence an equivalence relation. It follows similarly that $\pi_2(E)$ is an equivalence relation.
\end{proof}

\begin{lemma} 
\label{lemma: final}
Let $\C$ be a regular category in which every split monomorphism is coextensive. Then, given any product diagram $X_1 \xleftarrow{\pi_1} X \xrightarrow{\pi_2} X_2$ and any reflexive relation $R$ on $X$ we have $R = \pi_1(R) \times \pi_2(R)$.  
\end{lemma}
\begin{proof}
Suppose that $R$ is a reflexive relation on $X$ represented by $(R_0, r_1, r_2)$ and that $X \xrightarrow{d} R_0$ is the diagonal inclusion. Then we form the diagram 
\[\xymatrix{
X_1 \ar@/_3pc/[dd]_{1_{X_1}} \ar[d]_{d_1}  & \ar@{->>}[l]_{\pi_1} X  \ar@{->>}[r]^{\pi_2} \ar@{>->}[]+<0ex,-2.2ex>;[d]_-{d} & X_2 \ar[d]^{d_2} \ar@/^3pc/[dd]^{1_{X_2}}
\\
R^{(1)}_0 \pol \ar@{.>}@<0.5ex>[d]^-{r^{(1)}_2} \ar@{.>}@<-0.5ex>[d]_-{r^{(1)}_1} & \ar@{->>}[l]^{p_1} R_0 \ar@{->>}[r]_{p_2} \ar@<0.5ex>[d]^-{r_2} \ar@<-0.5ex>[d]_-{r_1} & R^{(2)}_0 \por \ar@{.>}@<0.5ex>[d]^-{r^{(2)}_2} \ar@{.>}@<-0.5ex>[d]_-{r^{(2)}_1} \\
X_1  & X \ar@{->>}[l]^-{\pi_1} \ar@{->>}[r]_-{\pi_2}& X_2  
}
\]
where the top squares are pushouts, and hence the middle row is a product diagram. Moreover, the dotted arrows are induced (reasonably) by the pushouts. By Lemma~\ref{lemma: f_1 f_2 mono} it will follow that $r^{(1)}_1, r^{(1)}_2$ and $r^{(2)}_1, r^{(2)}_2$ are jointly monomorphic. Every identity morphism is coextensive, thus all product projections in $\C$ are extremal (and therefore regular) epimorphisms by Proposition~\ref{proposition: coextensivity of identity = product projections are extremal},  therefore $p_1$ and $p_2$ are regular epimorphisms. From this it follows that $(R_0^{(1)},r^{(1)}_1, r^{(1)}_2)$ represents $\pi_1(R)$ and $(R_0^{(2)}, r^{(2)}_1, r^{(2)}_2)$ represents $\pi_2(R)$ so that $R = \pi_1(R) \times \pi_2(R)$. 
\end{proof}
\begin{remark}
In the context of a regular majority category \cite{Hoefnagel2020maj}, we have that given any product diagram $X_1 \xleftarrow{\pi_1} X \xrightarrow{\pi_2} X_2$ and any reflexive relation $R$ on $X$ we have $R = \pi_1(R) \times \pi_2(R)$. However, it is no true that every split monomorphism in a regular majority category is coextensive. For instance, the variety $\Lat$ of lattices is a regular majority category, which does not have coextensive split monomorphisms (as a result of Proposition~\ref{proposition: Barr-exact coextensive}). 
\end{remark}
\begin{proposition} \label{proposition: Barr-exact coextensive}
    Let $\C$ be a Barr-exact category. Then, every split monomorphism in $\C$ is coextensive if and only if $\C$ is coextensive. 
\end{proposition}
\begin{proof}
Suppose that every split monomorphism in $\C$ is coextensive. Note that each product projection is a regular epimorphism by Proposition~\ref{proposition: coextensivity of identity = product projections are extremal}. We show that every regular epimorphism is coextensive, and conclude the result by the dual of Corollary~\ref{corollary: C extensive iff split epis and inclusions are}. Note that we only need to show that every regular epimorphism satisfies $(\ExtFirstop)$ by Proposition~\ref{corollary: Ext1 are extensive}. To this end, let $q:X \to Y$ be any regular epimorphism in $\C$ and let $E=\Eq(q)$ be the equivalence relation represented by the kernel pair $(K,k_1,k_2)$ of $q$. Given any product diagram $\xymatrix{X_1 & X \ar@{->>}[r]_{\pi_2}  \ar@{->>}[l]^{\pi_1} & X_2}$ we consider  the diagram 
\[
\xymatrix{
K_1 \ar@<0.5ex>[d]^-{u_2} \ar@<-0.5ex>[d]_-{u_1} & \ar@{->>}[l]_{p_1} K \ar@{->>}[r]^{p_2} \ar@<0.5ex>[d]^-{k_2} \ar@<-0.5ex>[d]_-{k_1} & K_2  \ar@<0.5ex>[d]^-{v_2} \ar@<-0.5ex>[d]_-{v_1} \\
X_1 \ar@{->>}[d]_{q_1} & X \ar@{->>}[l]^-{\pi_1} \ar@{->>}[d]_{q} \ar@{->>}[r]_-{\pi_2}& X_2 \ar@{->>}[d]^{q_2} \\
Y_1 & \ar[l]^-{\alpha_1} Y \ar[r]_-{\alpha_2} & Y_2}
\]
where $(u_1, u_2)$ and $(v_1,v_2)$ are obtained as representatives of $\pi_1(E)$ and $\pi_2(E)$. By Lemma~\ref{lemma: equivalence relations under reg epi}, it follows that $\pi_1(E)$ and $\pi_2(E)$ are equivalence relations, so that they are effective. Hence, $(K_1,u_1,u_2)$ and $(K_2,v_1,v_2)$ are kernel pairs and therefore admit coequalisers $q_1$ and $q_2$ respectively. Then the morphisms $\alpha_1$ and $\alpha_2$ are induced by $q$ being the coequaliser of $k_1$ and $k_2$. Since $p_1$ and $p_2$ are regular epimorphisms, it follows that the squares are pushouts by Lemma~\ref{lemma: common-coequaliser}. Finally, since $E = \pi_1(E) \times \pi_2(E)$ by Lemma~\ref{lemma: final}, it follows that $(K,k_1,k_2)$ is a kernel pair of $q_1 \times q_2$. Since $q_1 \times q_2$ is a regular epimorphism, it is the coequaliser of $k_1$ and $k_2$. Therefore $q$ and $q_1 \times q_2$ are coequalisers of the same parallel pair, so that the bottom row is a product diagram. 
\end{proof}

\section*{Acknowledgements} Many thanks are due to Dr. F. van Niekerk for many interesting discussions on the topic of extensive morphisms.

\end{document}